\documentclass{article}
\usepackage{graphicx} % Required for inserting images

\title{Average pairwise Hamming distance of a subset in a hypercube}

\usepackage{amsmath, amssymb, amsthm, url, tikz, verbatim, hyperref}

\newtheorem{theorem}{Theorem}[section]

\newtheorem{lemma}[theorem]{Lemma}

\newtheorem{question}[theorem]{Question}

\newtheorem{conjecture}[theorem]{Conjecture}

\usetikzlibrary{calc}
\usetikzlibrary{patterns}
\usetikzlibrary{decorations.markings}
\usetikzlibrary{arrows,shapes.geometric,%
	decorations.pathreplacing,shapes,shadows}

\date{}

\author{}

\author{ Yijia Fang\thanks{Department of Mathematics, National University of Singapore. Email: fangyijia@u.nus.edu.}
\and
Hao Huang\thanks{Department of Mathematics, National University of Singapore. Email: huanghao@nus.edu.sg. Research supported in part by a start-up grant at NUS and an MOE Academic Research Fund (AcRF) Tier 1 grant A-8003627.}}

\begin{document}

\maketitle

\begin{abstract}
In this note, we prove that for a subset $\mathcal{F}$ of $\{0,1\}^n$ of density $\frac{1}{8}$, its average pairwise Hamming distance is at least $\frac{n-3}{2}$, with equality attained only by $(n-3)$-dimensional subcubes. This answers a question of Lei Yu, and solves an important open case of a more general question posed by Ahlswede and Katona in 1977.
\end{abstract}

\section{Introduction}
Given two $n$-dimensional binary vectors $x, y$, we denote by $d_H(x,y)$ the Hamming distance between them, which is the number of coordinates $i \in [n]$ such that $x_i \neq y_i$. A natural question was raised by Ahlswede and Katona \cite{AK77} about half a century ago: for a given $n$, among all subsets $\mathcal{F} \subset \{0,1\}^n$ of fixed size $|\mathcal{F}|$ or fixed density $\alpha=\frac{|\mathcal{F}|}{2^n}$, what is the minimum average Hamming distance between two independent random vectors, each uniformly chosen from $\mathcal{F}$? 

Denote by $\mathbb E_{X,Y\sim\mathrm{Unif}(\mathcal F)}~d_H(X,Y)$ this average distance.  It is not hard to observe that every $d$-dimensional subcube of $\{0,1\}^n$ has average pairwise distance $\frac{d}{2}$. Intuitively, to minimize the average pairwise distance of $\mathcal{F}$, we would want the vectors to be as close to each other as possible to form a ``compact'' structure. And a subcube is a natural candidate for such a ``compact'' construction. In light of this construction, one can ask whether for every $0 \le d \le n$, among all subsets $\mathcal{F} \subset \{0,1\}^n$ with $|\mathcal{F}|=2^d$, the average pairwise Hamming distance is always minimized by a $d$-dimensional subcube. This is trivially true for $d=n$. Alth\"ofer and Sillke \cite{AS92} showed that the average pairwise distance of every $\mathcal{F}$ is at least $\frac{n+1}{2}-\frac{2^{n-1}}{|\mathcal{F}|}$, which for $d=n-1$ (i.e. $|\mathcal{F}|=2^{n-1}$) shows the $(n-1)$-dimensional subcube is the minimizer. Let 
$$\beta (n, M)=\min\left\{ \mathbb E_{X,Y\sim\mathrm{Unif}(\mathcal F)}~d_H(X,Y): \mathcal{F}\subset \{0,1\}^n, |\mathcal{F}|=M\right \}.$$
So the results above give $\beta(n, 2^n)=\frac{n}{2}$ and $\beta(n, 2^{n-1})=\frac{n-1}{2}$. Fu, Wei, and Yeung \cite{FWY01} developed a linear-programming approach and proved $\beta(n, M) \ge \frac{n}{2}-\frac{2^{n-2}}{M}$ for $M \le 2^{n-1}$, which implies $\beta(n, 2^{n-2})=\frac{n-2}{2}$. 

However, when $d=n-5$ and $n=19$, there exists a construction (see Remark 4, Page 70 of \cite{A20}) showing $\beta(n, 2^{n-5})<\frac{n-5}{2}$ by taking all vectors of Hamming weight at most $4$ and some of the weight-$5$ vectors. Indeed calculations show that when the density is small, for sufficiently large $n$, a subset resembling a Hamming ball has average pairwise Hamming distance smaller than that of the subcube of equal size. There has also been substantial work on exact values, lower bounds, and asymptotic behavior of $\beta(n,M)$ in various parameter ranges, see  \cite{AA94, FWY01, JKM85, K02, M08, XF98, Y25, YT19}.

Even though the subcube is not always the minimizer, its optimality when the density is $\frac{1}{8}$ or $\frac{1}{16}$ remains unknown. In particular, one might ask whether $\beta(n, 2^{n-k})$ is still equal to $\frac{n-k}{2}$ for $k \in \{3, 4\}$. The $k=3$ case was explicitly posed by Yu (Problem $2$ in \cite{Y25}) in his recent breakthrough paper which improved the classical level-$1$ bound (Chang's inequality as in \cite{C02}) in both the small- and large-density regimes.

In this short note, we resolve the $k=3$ case and prove the following theorem, which shows that the subcube is still optimal when the density is $\frac{1}{8}$.

\begin{theorem}\label{thm_main}
Suppose $\mathcal{F} \subset \{0,1\}^n$ has density $\frac{1}{8}$, i.e. $|\mathcal{F}|=2^{n-3}$, then its average pairwise Hamming distance is at least $\frac{n-3}{2}$. Equality holds if and only if $\mathcal{F}$ is a $(n-3)$-dimensional subcube of $\{0,1\}^n$.
\end{theorem}

\section{The proof}
Let $f=1_{\mathcal{F}}$ be the indicator Boolean function of $\mathcal{F}$, then $\mathbb{E}[f]=\frac{|\mathcal{F}|}{2^n}=\frac{1}{8}$. Our next lemma is a standard application of Fourier analysis for Boolean functions, and shows that in order to lower bound the average pairwise distance in $\mathcal{F}$, it suffices to upper bound the level-$1$ Fourier weight $W_1[f]=\sum_{i=1}^n \widehat{f}(\{i\})^2$ of its indicator function. A more general version can be found in \cite{XF98}. To make our note self-contained, we include the proof for our special case below. We use the following Fourier convention throughout our paper.
$$\widehat{f}(S)=\frac{1}{2^n} \sum_{x \in\{0,1\}^n} f(x)(-1)^{\sum_{i \in S}x_i}.$$

\begin{lemma}\label{lem_reduction}
For $\mathcal{F} \subset \{0,1\}^n$ with $|\mathcal{F}|=2^{n-3}$, we have
$$\mathbb E_{X,Y\sim\mathrm{Unif}(\mathcal F)}~d_H(X,Y)
=\frac n2-32\sum_{i=1}^n\widehat f(\{i\})^2.
$$
\end{lemma}
\begin{proof}
For every $i \in [n]$, $$\widehat{f}(\{i\})=\frac{1}{2^n}\sum_{x \in \mathcal{F}}(-1)^{x_i}=\frac{|\mathcal{F}|}{2^n}\cdot \mathbb{E}[(-1)^{X_i}].$$
Using $|\mathcal{F}|=2^{n-3}$, we have
$$\mathbb{E}[(-1)^{X_i}]=8\widehat{f}(\{i\}).$$

Note that each $X_i, Y_i \in \{0,1\}$, so 
$$1_{\{X_i \neq Y_i\}}=\frac{1-(-1)^{X_i+Y_i}}{2}.$$
Therefore using the independence of $X$ and $Y$, we have
\begin{align*}
\mathbb E_{X,Y\sim\mathrm{Unif}(\mathcal F)}~d_H(X,Y)&=\sum_{i=1}^n \mathbb{E}[1_{\{X_i \neq Y_i\}}]=\sum_{i=1}^n \mathbb{E}\left[\frac{1-(-1)^{X_i+Y_i}}{2}\right]\\
&=\frac{n}{2}-\frac{1}{2} \sum_{i=1}^n \mathbb{E}\left[(-1)^{X_i}\right] \mathbb{E}\left[(-1)^{Y_i}\right]\\
&=\frac{n}{2}-\frac{1}{2}\sum_{i=1}^n \left(8\widehat{f}(\{i\})\right)^2\\
&=\frac{n}{2}-32\sum_{i=1}^n\widehat f(\{i\})^2.
\end{align*}
\end{proof}
One key estimate we use in our proof is the following inequality involving the cosine functions.
\begin{lemma}\label{lem_cos_easy}
Suppose $\theta_i \in [0, \frac{\pi}{2}]$ satisfy $\sum_i \theta_i^2=\frac{3\pi^2}{4}$, then 
$$12 \prod_i \cos \theta_i \le 1 + \prod_i \cos(2\theta_i).$$
Equality holds exactly when three $\theta_i$ are $\frac{\pi}{2}$ and the rest are $0$.
\end{lemma}
\begin{proof}
If the left hand side is non-positive, then the inequality automatically holds. If equality holds in this case, then $\prod_i \cos(2 \theta_i)=-1$, which means that every $\theta_i$ is either $0$ or $\frac{\pi}{2}$. Thus exactly three of them are $\frac{\pi}{2}$ and the rest are $0$.

Now suppose the left hand side is strictly positive. Using $\cos u \le e^{-u^2/2}$ for $u \in [0, \frac{\pi}{2}]$, we have 
$$12 \prod_i \cos \theta_i \le 12 e^{-\frac{\sum_i \theta_i^2}{2}}=12 e^{-3\pi^2/8}\approx 0.296\cdots <\frac{1}{2}.$$
So if $\prod_i \cos(2\theta_i) \ge -\frac{1}{2}$, then the strict inequality is also true.   

It remains to consider the case $\prod_i \cos(2\theta_i)<-\frac{1}{2}$. By the negativity of this product, there is an odd number of indices $i$ such that $\theta_i>\frac{\pi}{4}$. Suppose there is only one such index, then the sum of squares of the other $\theta_i$'s is at least $\frac{\pi^2}{2}$. The same $\cos u \le e^{-u^2/2}$ inequality gives
$$\prod_i \cos(2\theta_i) \ge -\prod_{i: \theta_i < \frac{\pi}{4}} \cos(2\theta_i) \ge - \exp\left(-2\sum_{i: \theta_i<\frac{\pi}{4}}\theta_i^2\right)\ge -e^{-\pi^2}>-\frac{1}{2},$$
which is impossible. Thus there are at least three $\theta_i$'s that are strictly greater than $\frac{\pi}{4}$. Without loss of generality, we assume $\theta_1, \theta_2, \theta_3 \in (\frac{\pi}{4}, \frac{\pi}{2}]$. They satisfy $\cos(2\theta_i) \in [-1, 0)$. By the AM-GM inequality,
$$\left(\prod_{i=1}^3 (1+\cos(2 \theta_i))\right)^{1/3} \le 1-\frac{\sum_{i=1}^3 (-\cos(2\theta_i))}{3}\le 1-\prod_{i=1}^3 (-\cos(2\theta_i))^{1/3}.$$
Using $\cos(\theta_i)=\sqrt{(1+\cos(2\theta_i))/2}$ and writing $q=\prod_{i=1}^3(-\cos(2\theta_i))^{1/3}$, we have
\begin{align*}
\frac{\prod_i \cos(\theta_i)}{1+\prod_i \cos(2\theta_i)} &\le \frac{\prod_{i=1}^3 \cos(\theta_i)}{1+\prod_{i=1}^3 \cos(2\theta_i)}\le \frac{(1-\prod_{i=1}^3(-\cos(2\theta_i))^{1/3})^{3/2}}{2\sqrt{2}(1-\prod_{i=1}^3(-\cos(2\theta_i)))}\\
&=\frac{(1-q)^{3/2}}{2\sqrt{2}(1-q^3)}=\frac{\sqrt{1-q}}{2\sqrt{2}(1+q+q^2)}. 
\end{align*}
Since we have assumed that $\prod_i \cos(2\theta_i)<-\frac{1}{2}$, this also gives $\prod_{i=1}^3 \cos(2\theta_i)<-\frac{1}{2}$. So $q>(1/2)^{1/3}>3/4$. Thus $\sqrt{1-q}<1/2$, and $1+q+q^2>1+(3/4)+(3/4)^2>9/4$. Plugged into the inequality above, we have
$$\frac{\prod_i \cos(\theta_i)}{1+\prod_i \cos(2\theta_i)} \le \frac{\frac{1}{2}}{2\sqrt{2}\cdot \frac{9}{4}}=\frac{1}{9\sqrt{2}}<\frac{1}{12},$$
completing the proof of the last case.
\end{proof}

\begin{lemma}\label{lem_cos_hard}
For a real number $u$, write $u_+=\max\{u, 0\}$.
Define $$\Phi(u)=u^2-1+\frac{4}{\pi}\cos\frac{\pi u}{2}-\frac{2}{3\pi}\cos^2 \frac{\pi u}{2},$$
then for every $u \in \mathbb{R}$,
$$\Phi(u) \ge 8(|u|-2)_+.$$
Equality occurs exactly when $|u| \in \{1,3,5\}.$
\end{lemma}
\begin{proof}
First note that $\Phi(u)$ and $8(|u|-2)_+$ are both even functions. So we just need to prove the inequality $\Phi(u) \ge 8(u-2)_+$ for $u \ge 0$ and show that equality is attained only at $u=1,3,5$.

Secondly, the periodicity of the cosine function gives that for every $u$,
$$\Phi(u)-\Phi(u-4)=8(u-2).$$
For $u \in [0,2]$, the inequality $\Phi(u) \ge 8(u-2)_+$ just becomes $\Phi(u) \ge 0$. For $u \ge 2$, it suffices to show $\Phi(u-4) \ge 0$. So if we are able to establish that $\Phi(u) \ge 0$ for all $u \ge 0$ and only attains $0$  at $u=1$ in the interval $[0, \infty)$, then by the fact that $\Phi$ is even, we immediately complete the proof and the characterization of the equality case.

We do standard calculus for the last step. By calculation $\Phi(1)=\Phi'(0)=\Phi'(1)=0$. We discuss two cases. First for $u \in [0,1]$, direct differentiation gives
$$\Phi''(u)=2-\pi \cos(\pi u/2)+\frac{2\pi}{3}\cos^2(\pi u/2)-\pi/3,$$
which is a quadratic function in $\cos(\pi u/2) \in [0,1]$. At the endpoints, $\Phi''(0)=2-2\pi/3<0$ and $\Phi''(1)=2-\pi/3>0$.
%At the two ends $\Phi''$ are $2-\pi/3>0$ and $2-2\pi/3<0$. 
So in $[0,1]$, $\Phi'(u)$ first decreases and then increases. Using $\Phi'(0)=\Phi'(1)=0$, we have $\Phi'(u) < 0$ on $(0,1)$. Therefore $\Phi(u)>\Phi(1)=0$ for $u \in [0, 1)$.

For $u \ge 1$, we substitute $v=\pi(u-1)/2 \ge 0$, then 
$$\frac{\pi}{4}\Phi(u)=v-\sin v+\frac{v^2}{\pi}-\frac{\sin^2 v}{6} \ge \frac{v^2}{\pi}-\frac{\sin^2 v}{6} \ge 0.$$
For both  inequalities, we use $v \ge |\sin v|$.
\end{proof}

\begin{lemma}\label{lem_sign_sum}
Let $\varepsilon_1, \cdots, \varepsilon_n$ be independent uniform signs from $\{-1, 1\}$. Suppose $|a_i| \le 1$, and $\sum_{i=1}^n a_i^2=3$. Then
$$\mathbb{E}\left(\left|\sum_{i=1}^n \varepsilon_i a_i\right|-2\right)_{+} \le \frac{1}{4}.$$
Equality holds if and only if three of the $a_i$'s are $\pm 1$, and the rest are $0$.
\end{lemma}
\begin{proof}
Without loss of generality we can assume all $a_i \in [0,1]$ by replacing $a_i$ by $|a_i|$.
We let $X=\sum_{i=1}^n \varepsilon_i a_i$ and $\theta_i=\frac{\pi a_i}{2} \in [0, \frac{\pi}{2}]$. Then 
$$\mathbb{E}[X^2]=\sum_{i=1}^n a_i^2 = 3,\qquad \sum_{i=1}^n \theta_i^2=\frac{\pi^2}{4}\sum_{i=1}^n a_i^2=\frac{3\pi^2}{4}.$$
By the independence of $\varepsilon_i$'s, we have
$$\mathbb{E}[\cos(\pi X)]=\textup{Re}\left[\mathbb{E}[e^{i\pi X}]\right]=\textup{Re}\left[\prod_{i=1}^n \mathbb{E}[e^{i\pi \varepsilon_i a_i}]\right]=\prod_{i=1}^n \cos(2\theta_i).$$
Similarly $$\mathbb{E}\left[\cos\left(\frac{\pi X}{2}\right)\right]=\prod_{i=1}^n \cos(\theta_i).$$
Applying Lemma \ref{lem_cos_hard} with $u=X$ and taking expectations, we have
\begin{align*}
8\mathbb{E}[(|X|-2)_+] &\le \mathbb{E}[\Phi(X)]=\mathbb{E}[X^2]-1+\frac{4}{\pi}\mathbb{E}\left[\cos \frac{\pi X}{2}\right]-\frac{1}{3\pi}\mathbb{E}[1+\cos \pi X]\\
&=2+\frac{1}{3\pi}\left(12 \prod_{i=1}^n \cos(\theta_i)-1-\prod_{i=1}^n \cos(2\theta_i)\right) \le 2.
\end{align*}
Dividing both sides by $8$ concludes the proof of the inequality. When equality is attained, we need equality in Lemma \ref{lem_cos_easy}, which gives three $\theta_i=\frac{\pi}{2}$ and the rest equal $0$. Thus exactly three $a_i$'s have absolute value $1$, and all the others vanish. 
\end{proof}

The celebrated level-$1$ bound, often called Chang's inequality \cite{C02}, states that $W_1[f] \le 2\mathbb{E}[f]^2 \ln \frac{1}{\mathbb{E}[f]}$. This specific form can be found in \cite{IMR14}. For our case $\mathbb{E}[f]=\frac{1}{8}$, this gives $W_1[f] \le \frac{3 \ln 2}{32} \approx 0.0649825$.
This is not enough for establishing the optimality of the $(n-3)$-dimensional subcube. We prove the following stronger bound.
\begin{lemma}\label{lem_main}
For $f=1_{\mathcal{F}}$ with $\mathbb{E}[f]=\frac{1}{8}$, we have 
$$W_1[f]=\sum_{i=1}^n \widehat{f}(\{i\})^2 \le \frac{3}{64}=0.046875.$$
Furthermore when equality holds, $\mathcal{F}$ is a $(n-3)$-dimensional subcube of $\{0,1\}^n$.
\end{lemma}
\begin{proof}
Let us start by assuming $M=\sum_{i=1}^n \widehat{f}(\{i\})^2 \ge \frac{3}{64}$. We will show that equality is the only possibility.

Note that $$|\widehat{f}(\{i\})|=\left|\frac{1}{2^n} \sum_{z \in \mathcal{F}} (-1)^{z_i}\right| \le \frac{|\mathcal{F}|}{2^n}=\frac{1}{8}. $$
Therefore if we let $$a_i=\sqrt{3} \cdot \frac{\widehat{f}(\{i\})}{\sqrt{\sum_{i=1}^n \widehat{f}(\{i\})^2}},$$
we have $|a_i| \le \sqrt{3} \cdot \frac{1/8}{\sqrt{3/64}}=1$ and $\sum_{i=1}^n a_i^2=3$. So the assumptions of Lemma \ref{lem_sign_sum} are satisfied. Therefore if we let $X=\sum_{i=1}^n x_i a_i$ for $x=(x_1, \cdots, x_n)$ uniformly chosen from $\{-1, 1\}^n$, then
$$\mathbb{E}(|X|-2)_+ \le \frac{1}{4}.$$
Note that we can let $x_i=(-1)^{z_i}$, with $z_i$ uniform from $\{0,1\}$,
$$\mathbb{E}_{x}[f\cdot X]=\mathbb{E}_{z}[f\cdot X]=\sum_{i=1}^n a_i \mathbb{E}_z[f(z_1,\cdots, z_n)\cdot (-1)^{z_i}]=\sum_{i=1}^n a_i \widehat{f}(\{i\})=\sqrt{3(\sum_{i=1}^n \widehat{f}(\{i\})^2)}=\sqrt{3M}.$$
For each $z=(z_1, \cdots, z_n)$, clearly we have the pointwise inequality $f(z_1,\cdots, z_n)(X-2) \le (X-2)_{+}$. Taking expectations, and using that $X$ is symmetric about zero, we have
\begin{align*}
\sqrt{3M}&=\mathbb{E}_z[f \cdot X]=\mathbb{E}_z[f(X-2)] + 2\mathbb{E}[f]\\
&\le \mathbb{E}_x[(X-2)_+]+ \frac{1}{4}\\
&=\frac{1}{2}\mathbb{E}_x[(|X|-2)_+]+\frac{1}{4} \le \frac{1}{2}\cdot \frac{1}{4}+\frac{1}{4}=\frac{3}{8}.
\end{align*}
This gives 
$$\sum_{i=1}^n \widehat{f}(\{i\})^2 = M \le \frac{3}{64},$$
completing the proof of the inequality.

When the equality holds, we have $\mathbb{E}_x[(|X|-2)_+]=\frac{1}{4}$. Using the equality condition from Lemma \ref{lem_sign_sum}, for $$a_i=\sqrt{3} \cdot \frac{\widehat{f}(\{i\})}{\sqrt{\sum_{i=1}^n \widehat{f}(\{i\})^2}},$$
exactly three of them are $\pm 1$, and the rest are $0$. Without loss of generality, we can assume $\widehat{f}(\{1\}), \widehat{f}(\{2\}), \widehat{f}(\{3\}) \in \{-\frac{1}{8}, \frac{1}{8}\}$ and the others are zero. This immediately gives $\mathbb{E}_{X \sim\mathrm{Unif}(\mathcal F)}[(-1)^{X_i}] \in \{-1, 1\}$ for $i=1,2,3$. This clearly shows that $\mathcal{F}$ is a $(n-3)$-dimensional subcube of $\{0,1\}^n$ consisting of vectors whose first three coordinates are fixed.
\end{proof}

Now it is straightforward to combine the previous lemmas and complete the proof of our main theorem.

\begin{proof}[Proof of Theorem \ref{thm_main}]
For $|\mathcal{F}|=2^{n-3}$, $f=1_{\mathcal{F}}$ satisfies the conditions in Lemma \ref{lem_main}, thus
$$\sum_{i=1}^n \widehat{f}(\{i\})^2 \le \frac{3}{64}.$$
Applying Lemma \ref{lem_reduction}, we have
$$\mathbb E_{X,Y\sim\mathrm{Unif}(\mathcal F)}~d_H(X,Y)
=\frac n2-32\sum_{i=1}^n\widehat f(\{i\})^2 \ge \frac{n}{2}-32 \cdot \frac{3}{64} = \frac{n-3}{2}.
$$
This shows the average Hamming distance is at least $\frac{n-3}{2}$. The equality characterization also follows from Lemma \ref{lem_main}.
\end{proof}

\section{Concluding Remarks}
We believe the same technique should also lead to the following result for the density $\frac{1}{16}$ case: for $\mathcal{F} \subset \{0,1\}^n$ and $|\mathcal{F}|=2^{n-4}$, the average pairwise Hamming distance of $\mathcal{F}$ is at least $\frac{n-4}{2}$. In particular, we believe the following conjecture, an analogue of Lemma \ref{lem_sign_sum}, is true.
\begin{conjecture}
Let $\varepsilon_1, \cdots, \varepsilon_n$ be independent uniform signs from $\{-1, 1\}$. Suppose $|a_i| \le 1$, and $\sum_{i=1}^n a_i^2=4$. Then
$$\mathbb{E}\left(\left|\sum_{i=1}^n \varepsilon_i a_i\right|-3\right)_{+} \le \frac{1}{8}.$$
Equality holds if and only if exactly four of the $a_i$'s are $\pm 1$, and the rest are $0$.
\end{conjecture}
We believe it can be proved similarly by constructing a function $\Phi(u)$ using higher degree polynomials in the cosine functions, such that $\Phi(u)$ is comparable to  the positive part of $|u|-3$, as we have done in Lemma \ref{lem_cos_hard}. And it would immediately imply the desired result for density $\frac{1}{16}$.

Below is another natural question: for given $n$, and a sequence $a_0, \cdots, a_n$ of real numbers, given a subset $\mathcal{F} \subset \{0,1\}^n$ of fixed size, we may consider the weighted sum
$$W(\mathcal{F})=\sum_{x \in \mathcal{F}} \sum_{y \in \mathcal{F}} a_{d_H(x,y)},$$
and ask which $\mathcal{F}$ minimizes $W(\mathcal{F})$ if the dimension $n$ and the size $|\mathcal{F}|$ are fixed. For example, when $a_i=i$ this corresponds to the average pairwise Hamming distance problem studied in this paper. When $a_1=-1$ and all other $a_j=0$, this is exactly the edge isoperimetric inequality for the hypercube. From these examples, we may suggest the following problem:

\begin{question}
Is it true that if a sequence of non-negative real numbers $\{a_i\}$ is non-decreasing and concave, then for fixed $M \in \{2^{n-4}, 2^{n-3}, 2^{n-2}, 2^{n-1}\}$, the corresponding subcube construction always minimizes $W(\mathcal{F})$, among all $\mathcal{F} \subset \{0,1\}^n$ with $|\mathcal{F}|=M$?
\end{question}

One can also ask the following question, inspired by the edge isoperimetric inequality:
\begin{question}
Which infinite sequences $\{a_i\}_{i=0,1,2,\cdots }$ of constants, independent of $n$, have the following property: for every $n$ and every  $d \in \{0,1, \cdots, n\}$, the $d$-dimensional subcube always minimizes $W(\mathcal{F})$, among all $\mathcal{F} \subset \{0,1\}^n$ with $|\mathcal{F}|=2^d$?
\end{question}
\medskip

\noindent \textbf{Declaration of AI use.} ChatGPT 6 Astra was used to explore trigonometric polynomials that lead to the proof of Lemma \ref{lem_cos_hard}, as well as proofreading and language polishing. All AI-assisted outputs were independently checked, verified, and simplified by the authors, who assume full responsibility for the content of this work.\\

\noindent \textbf{Acknowledgment. }We are deeply grateful to Fan Chang and Lei Yu for bringing this elegant conjecture to our attention. Their insights and suggestions during the early stages of this project played an important role in shaping the direction of this work.

\end{document}